\documentclass{article}
\usepackage{amssymb}
\usepackage{amsmath}
\usepackage{amsfonts}
\usepackage{amscd}
\usepackage{epsfig}
\usepackage{amsmath,amstext,amsthm,amsbsy,amssymb}

\newtheorem{thm}{Theorem}[section]
\newtheorem{lem}{Lemma}[section]
\newtheorem{pro}{Proposition}[section]
\newtheorem{cor}{Corollary}[section]

\newtheorem{defi}{Definition}[section]
\newtheorem{exam}{Example}[section]
\newcommand{\bi}{{\bf i}}

\newcommand{\be}{{\bf e}}

\newcommand{\bc}{{\mathbb C}}
\newcommand{\bh}{{\mathbb H}}
\newcommand{\br}{{\mathbb R}}

\begin{document}

\title{The Moore-Penrose Inverses of Clifford Algebra $C\ell_{1,2}$}
\author{Wensheng Cao,  Ronglan Zheng, Huihui Cao}
\date{}
\maketitle

\bigskip
{\bf Abstract.} \,\,  In this paper, we introduce a ring  isomorphism between the  Clifford algebra $C\ell_{1,2}$ and a ring of matrices. By  such a ring isomorphism, we  introduce the concept of the Moore-Penrose inverse in  Clifford algebra $C\ell_{1,2}$.  Using the Moore-Penrose inverse, we  solve the linear equation $axb=d$ in $C\ell_{1,2}$.  We also obtain necessary and sufficient conditions for two numbers in $C\ell_{1,2}$  to be similar.

\vspace{1mm}\baselineskip 12pt



\section{Introduction}

\qquad Clifford's geometric algebras were firstly created by William K. Clifford in 1878. Such algebras were independently rediscovered by Lipschitz in 1880 \cite{lou}. Clifford algebras turned out to play an important role in applications to theoretical physics and some related fields \cite{lib11,lou}.  In this paper we adopt the following definition.
\begin{defi} The Clifford algebra  $C\ell_{p,q}$ with $p+q=n$  is generated by the orthonormal basis $\{i_1,\cdots,i_n\}$ of $\br^{p,q}$ with  the multiplication rules \cite{lou}
\begin{equation}\label{mulrul}i_t^2=1,\,1\le t\le p,\,\,\,\,\,i_t^2=-1,\,p<t\le n,\,\,\,\,\,i_ti_m=-i_mi_t,\,t<m.\end{equation}
\end{defi}
The Clifford algebra  $C\ell_{p,q}$  can be thought of as a  $2^n$-dimensional real linear space.  Let $\br$, $\bc$, $\bh$ and $\bh_s$  be  respectively the real numbers, the complex numbers, the quaternions and the split quaternions.  Then we have $\bc\cong C\ell_{0,1}$, $\bh\cong C\ell_{0,2}$ and $\bh_s\cong C\ell_{1,1}$.

 In this paper we focus on the Clifford algebra $C\ell_{1,2}$.   The basis of $C\ell_{1,2}$ over $\br$ is
\begin{equation}\label{basise}
\be_0=1,\,\be_1=i_1,\,\be_2=i_2,\,\be_3=i_1i_2,\,\be_4=i_3,\,\be_5=i_1i_3,\,\be_6=i_2i_3,\,\be_7=i_1i_2i_3.
\end{equation}
	According to the multiplication rules (\ref{mulrul}), we have the following multiplication rules  for the basis of $C\ell_{1,2}$.
\begin{table}[h]\label{tab1}
	\centering
	\caption{Multiplication table for the Clifford algebra $C\ell_{1,2}$}
	\begin{tabular}{c|rrrrrrr}
		& $\be_1$ &  $\be_2$  &  $\be_3$& $\be_4$ &  $\be_5$ &  $\be_6$& $\be_7$ \\
		\hline
		$\be_1$ &1 & $\be_3$ & $\be_2$ & $\be_5$ & $\be_4$ & $\be_7$ & $\be_6$ \\
		$\be_2$& -$\be_3$ &-1& $\be_1$ & $\be_6$ & -$\be_7$ & -$\be_4$ & $\be_5$ \\
		$\be_3$& -$\be_2$ & -$\be_1$ & 1 & $\be_7$ & -$\be_6$ & -$\be_5$ & $\be_4$ \\
		$\be_4$& -$\be_5$ & -$\be_6$ & $\be_7$ & -1 & $\be_1$ & $\be_2$ & -$\be_3$ \\
		$\be_5$& -$\be_4$ & -$\be_7$ & $\be_6$ & -$\be_1$ &1 & $\be_3$ & -$\be_2$ \\
		$\be_6$& $\be_7$ & $\be_4$ & $\be_5$ & -$\be_2$ & -$\be_3$ & -1 & -$\be_1$ \\
		$\be_7$& $\be_6$ & $\be_5$ & $\be_4$ & -$\be_3$ & -$\be_2$ & -$\be_1$ & -1
	\end{tabular}
\end{table}

Let \begin{equation}\label{cent}Cent(C\ell_{1,2})=\{a\in C\ell_{1,2}:xa=ax,\forall x\in C\ell_{1,2}\}.\end{equation} Obviously, we have the following proposition.
\begin{pro}\label{subalgebra}
	\begin{itemize}
		\item[(1)] 	$C\ell_{1,1}=span\{\be_0,\be_1,\be_2,\be_3\}$ and
		  \begin{equation} C\ell_{1,2}=C\ell_{1,1}+C\ell_{1,1}\be_4\end{equation}
		 and therefore each $a\in C\ell_{1,2}$ can be represented by
		\begin{equation}
			a=a_h+a_H\be_4,
		\end{equation}
		where $a_h=a_0+a_1\be_1+a_2\be_2+a_3\be_3,a_H=a_4+a_5\be_1+a_6\be_2+a_7\be_3\in C\ell_{1,1}.$
		\item[(2)] The center of $C\ell_{1,2}$ is  	\begin{equation}Cent(C\ell_{1,2})=\br+\br \be_7.	\end{equation}
	\end{itemize}
		\end{pro}

\begin{defi}
For $a=a_0+a_1\be_1+a_2\be_2+a_3\be_3+a_4\be_4+a_5\be_5+a_6\be_6+a_7\be_7\in
C\ell_{1,2}$ where $a_i\in \br$, we define the following associated notations of $a$:

the conjugate of $a$:\ \  $
\bar{a}=a_0-a_1\be_1-a_2\be_2-a_3\be_3-a_4\be_4-a_5\be_5-a_6\be_6+a_7\be_7;$

the prime of $a$:\ \ $a'=a_0+a_1\be_1-a_2\be_2+a_3\be_3-a_4\be_4+a_5\be_5-a_6\be_6-a_7\be_7;$

the real  part of $a$ :\ \ $Cre(a)=\frac{1}{2}(a+\bar{a})=a_0+a_7e_7;$

the imaginary part of $a$ :\ \
$Cim(a)=a-Cre(a);$

 $N(a)=a_0^2-a_1^2+a_2^2-a_3^2+a_4^2-a_5^2+a_6^2-a_7^2;$

 $T(a)=a_0a_7+a_2a_5-a_1a_6-a_3a_4;$

$P(a)=N(a)^2+4T(a)^2.$

\end{defi}

We can verify the following proposition by applying the
multiplication rules in Table \ref{tab1}.

\begin{pro}\label{properties} Let $a,b\in C\ell_{1,2}$. Then
\begin{itemize}
	\item[(1)]$T(a)=T(\bar{a})=-T(a'),\,\,\, N(a)=N(a')=N(\bar{a});$
	\item[(2)] $\overline{ab}=\bar{b}\bar{a},\,\,\,(ab)'=b'a';$
	\item[(3)] $a\bar{a}=\bar{a}a=N(a)+2T(a)\be_7;$
	\item[(4)] $T(ab)=N(a)T(b)+N(b)T(a);$
	\item[(5)] $N(ab)=N(a)N(b)-4T(a)T(b);$
	\item[(6)] $a\bar{a}(N(a)-2T(a)\be_7)=a(N(a)-2T(a)\be_7)\bar{a}=P(a);$
	\item[(7)] $P(ab)=P(a)P(b);$
	\item[(8)] $Cre(ab)=Cre(ba).$
	\end{itemize}
\end{pro}

Algebra isomorphisms are useful tools in studying the properties of Clifford algebra. Using
the real matrix representation, Cao \cite{cao} obtained  the Moore-Penrose inverse of $C\ell_{0,3}$ and  studied the  similarity and consimilarity in  $C\ell_{0,3}$.
  Ablamowicz \cite{abla},  Cao and Chang  \cite{caochang} used different algebra isomorphisms  to find  the Moore-Penrose inverse of split quaternions, which can be thought of as $C\ell_{1,1}$.

In this paper, we focus on the concepts of the Moore-Penrose inverse  and similarity  in  Clifford algebra $C\ell_{1,2}$.

The paper is organized as follows. In Section \ref{matrixrep}, we will introduce a ring isomorphism between the  Clifford algebra $C\ell_{1,2}$ and a ring of matrices.  By such a ring isomorphism, we will represent $a\in C\ell_{1,2}$ by a matrix $L(a)$.  We obtain the determinant and the eigenvalues of $L(a)$.   In Section \ref{mpinvsec}, by  such a ring isomorphism, we  introduce the concept of the Moore-Penrose inverse in  Clifford algebra $C\ell_{1,2}$.  Section \ref{eqsec} aims to solve the linear equation $axb=d$ in $C\ell_{1,2}$. In Section \ref{simsec},  we will obtain some necessary and sufficient conditions for two numbers in $C\ell_{1,2}$ to be similar.

\section{Ring isomorphism}\label{matrixrep}

\qquad Let $A^T$ be  the transpose of matrix $A$. Denote $\overrightarrow{x}=(x_0,x_1,x_2,x_3,x_4,x_5,x_6,x_7)^T\in \br^8$ for
$x=x_0+x_1\be_1+x_2\be_2+x_3\be_3+x_4\be_4+x_5\be_5+x_6\be_6+x_7\be_7\in C\ell_{1,2}$. Each $a\in  C\ell_{1,2}$ define two maps from  $C\ell_{1,2}$ to $C\ell_{1,2}$ by
\begin{equation}\label{isomorph1}
	L_a:x\to ax
\end{equation}
and
\begin{equation}\label{isomorph2}
	R_a:x\to xa.
\end{equation}
Based on  multiplication rules in  Table \ref{tab1}, multiplication can be represented by an ordinary matrix-by-vector product. Such a method has been used to study some properties of quaternions \cite{gros}.  In such a way, we can verify the  following two propositions.
\begin{pro} \label{promat}
For $a,x\in C\ell_{1,2}$, we  have
\begin{equation}
	\overrightarrow{ax}=L(a)\overrightarrow{x}
\end{equation}
and
\begin{equation}
	\overrightarrow{xa}=R(a)\overrightarrow{x},
\end{equation}
where
\begin{equation} \label{la}
	L(a)=\left(
	\begin{array}{cccccccc}
		a_0 & a_1  & -a_2  & a_3 & -a_4 & a_5  & -a_6  & -a_7  \\
		a_1 & a_0   & -a_3  & a_2  & -a_5 & a_4   & -a_7  & -a_6  \\
		a_2 & -a_3  & a_0   & a_1  & -a_6 & -a_7  & a_4   & -a_5 \\
		a_3 &  -a_2 & a_1   & a_0 & -a_7 & -a_6 & a_5   & -a_4 \\
		a_4 & -a_5  & a_6  & a_7 & a_0 & a_1  & -a_2  & a_3  \\
		a_5 & -a_4   & a_7  & a_6  & a_1 &a_0   & -a_3  & a_2  \\
		a_6 & a_7  & -a_4   &a_5   & a_2 & -a_3  & a_0   & a_1 \\
		a_7 & a_6 & -a_5   & a_4 &a_3 &  -a_2 & a_1   & a_0
	\end{array}
	\right)
\end{equation}
and
\begin{equation} \label{ra}
	R(a)=K_8 L(a)^TK_8, \end{equation}
where $K_8=diag(1,1,-1,1,-1,1,-1,-1)$.
\end{pro}

\begin{pro}Let $S_8=diag(1,-1,1,-1,1,-1,1,-1)$. Then we have
\begin{equation}L(a')=L(a)^T;
\end{equation}
\begin{equation} R(a')=R(a)^T;
\end{equation}
\begin{equation} \label{bara} L(\bar{a})=S_8L(a)^TS_8; \end{equation}
\begin{equation} \label{Rbara} R(\bar{a})=S_8R(a)^TS_8. \end{equation}
\end{pro}

\begin{pro}\label{pro2.2} Let $E_n$ be the identity matrix of order $n$.
 For $a,b\in C\ell_{1,2}, \,\lambda\in \br$, we have
\begin{itemize}
  \item[(1)]  $a=b\Longleftrightarrow L(a)=L(b)\Longleftrightarrow R(a)=R(b),\, L(\be_0)=R(\be_0)=E_8;$
 \item[(2)] $L(a+b)=L(a)+L(b),\, L(\lambda a)=\lambda L(a);$
  \item[(3)]$ R(a+b)=R(a)+R(b), \,R(\lambda a)=\lambda R(a);$
\item[(4)]  $ L(a)R(b)=R(b)L(a), \,R(ab)=R(b)R(a),\,L(ab)=L(a)L(b).$
\end{itemize}
\end{pro}
\begin{proof}  By Proposition \ref{promat}, we have  $L(ab)\overrightarrow{x}=\overrightarrow{abx}=L(a)\overrightarrow{bx}=L(a)L(b)\overrightarrow{x}$. Hence $L(ab)=L(a)L(b)$. Similarly,  $R(ab)\overrightarrow{x}=\overrightarrow{xab}=R(b)\overrightarrow{xa}=R(b)R(a)\overrightarrow{x}$ infers that
$ R(ab)=R(b)R(a)$. By $\overrightarrow{axb}=L(a)\overrightarrow{xb}=L(a)R(b)\overrightarrow{x}$ and
$\overrightarrow{axb}=R(b)\overrightarrow{ax}=R(b)L(a)\overrightarrow{x}$, we have $L(a)R(b)=R(b)L(a)$.  The other results are obvious.
\end{proof}

Let $Mat(8,\br)$ be the set of real matrices of order 8. By Proposition \ref{pro2.2}, $C\ell_{1,2}$ and $Mat(8,\br)$ can be thought of as the rings $(C\ell_{1,2},+,\cdot)$ and  $( Mat(8,\br),+,\cdot)$. Then we have the following theorem. \begin{thm}\label{ringiso}
Denote the map $$L:C\ell_{1,2}\to Mat(8,\br)$$ by
 \begin{equation}\label{isomorphmat}
	L: a\to L(a).
\end{equation}
Then $L$ is a ring homomorphism from $(C\ell_{1,2},+,\cdot)$ to $(Mat(8,\br),+,\cdot)$. Especially, let $im(L)$ be the image of such a homomorphism. Then $$L:C\ell_{1,2}\to im(L)$$ is a ring isomorphism.
	\end{thm}

\begin{pro}
	$\det(L(a))=\det(R(a))=(N(a)^2+4T(a)^2)^2=P(a)^2.$
\end{pro}
\begin{proof} Let
	$M=\left(
	\begin{array}{cccc}
		0 &0   &0   & 1\\
		0 &  0& 1 &0 \\
		0 & 1 & 0 & 0\\
		1&0  & 0&0
	\end{array}\right)$.  Then
\begin{equation}\label{le7} L(\be_7)=\left(
	\begin{array}{cc}
		0 & -M \\
		M & 0
	\end{array}\right),\ \,  M^2=E_4.\end{equation}    By Proposition \ref{properties}, we have
	$$L(a\bar{a})=L(N(a)+2T(a)\be_7)=\left(
	\begin{array}{cc}
		N(a)E_4 & -2T(a)M \\
		2T(a)M &N(a)E_4
	\end{array}\right).$$
	Thus $$\det(L(a\bar{a}))=(N(a)^2+4T(a)^2)^4=P(a)^4.$$
		   By Proposition \ref{pro2.2}, we have  $$L(a\bar{a})=L(a)L(\bar{a})=L(a)S_8L(a)^TS_8.$$ Hence $\det(L(a\bar{a}))=\det(L(a))^2$ and therefore  $\det(L(a))=P(a)^2$.
		   Since 	$R(a)=K_8 L(a)^TK_8$, we have $\det(L(a))=\det(R(a)).$
\end{proof}

\begin{pro}\label{proegi}
	The eigenvalues of  $L(a)$
	are given by
	
$$ \lambda_{1,2}=a_0+a_7\bi\pm
\sqrt{(a_0+a_7\bi)^2-N(a)-2T(a)\bi}
$$
and
$$ \lambda_{3,4}=a_0-a_7\bi\pm
\sqrt{(a_0-a_7\bi)^2-N(a)+2T(a)\bi}
$$	
	where $\bi=\sqrt{-1}\in \bc$  and  each eigenvalue occurs with algebraic
	multiplicity 2.
\end{pro}

\begin{proof} Let $\lambda$ be an eigenvalue of $L(a)$. Then
$\det(\lambda E_8-L(a))=0$.  Note that  $L(\bar{a})=S_8L(a)^TS_8$. Therefore such a  $\lambda$ is also a eigenvalue of $L(\bar{a})$, i.e.  $\det(\lambda
E_8-L(\bar{a}))=0$.
Let $$U=(\lambda
E_8-L(a))(\lambda E_8-L(\bar{a})).$$
Then
$$\det(U)=0.$$
 Note that
$$U=\lambda^2
E_8-\lambda L(a+\bar{a})+L(a\bar{a})=(\lambda^2-2\lambda
a_0+N(a))E_8+2(T(a)-\lambda a_7)L(\be_7),$$
that is
$$U=\left(
\begin{array}{cc}
	(\lambda^2-2\lambda
	a_0+N(a))E_4 &-2(T(a)-\lambda a_7)M\\
2(T(a)-\lambda a_7)M& (\lambda^2-2\lambda
a_0+N(a))E_4
\end{array}\right).$$
Hence $$UU^T=[(\lambda^2-2\lambda
a_0+N(a))^2+4(T(a)-\lambda a_7)^2]E_8.$$
Therefore
 \begin{equation} (\lambda^2-2\lambda
 	a_0+N(a))^2+4(T(a)-\lambda a_7)^2=0.\end{equation}
  Thus
\begin{equation} \lambda^2-2\lambda(a_0+a_7\bi)+N(a)+2T(a)\bi=0\end{equation} or
\begin{equation} \lambda^2-2\lambda(a_0-a_7\bi)+N(a)-2T(a)\bi=0.\end{equation}
The
solutions of the above two equations are the eigenvalues of
$L(a)$ and each eigenvalue  occurs with algebraic multiplicity
2.
\end{proof}

\begin{exam}
    Let $a=1-\be_1+\be_2+\be_3-\be_7$. Then $N(a)=-1$, $T(a)=-1$ and
    $$L(a)=\left(
	\begin{array}{cccccccc}
		1 & -1  & -1  & 1  & 0 & 0  &  0  & 1  \\
	   -1 &  1  & -1  & 1  & 0 & 0  &  1  & 0  \\
		1 & -1  &  1  &-1  & 0 & 1  &  0  & 0 \\
		1 & -1  & -1  & 1  & 1 & 0  &  0  & 0 \\
		0 &  0  &  0  &-1  & 1 &-1  & -1  & 1  \\
		0 &  0  & -1  & 0  &-1 & 1  & -1  & 1  \\
		0 & -1  &  0  & 0  & 1 &-1  &  1  &-1 \\
	   -1 &  0  &  0  & 0  & 1 &-1  & -1  & 1
	\end{array}
	\right).$$
The solutions of $det(\lambda E_8-L(a))=0$ are given by
$$\lambda_1=2-\bi,\,\,\lambda_2=-\bi,\,\,\lambda_3=2+\bi,\,\,\lambda_4=\bi,$$
where each eigenvalue occurs with algebraic multiplicity 2.
\end{exam}

\begin{pro}\label{uinverse}  Let $a,b,c\in C\ell_{1,2}$. Then
	 \begin{itemize}
	 	\item[(1)] If $ab=1$ then  $ba=1$;
	 	\item[(2)]  If $ab=1$ and $ac=1$ then $b=c$.
	 \end{itemize}
	 	 \end{pro}
\begin{proof}
By Theorem \ref{ringiso}, if $ab=1$ then $L(a)L(b)=E_8$. So $L(a)$ is invertible.  Hence we have $L(b)L(a)=E_8$, which implies that $ba=1$.  Also, if $ab=1$ and $ac=1$ then we have $a(b-c)=0$ and $L(a)(L(b)-L(c))=0$, which implies that $b=c$.
\end{proof}

\begin{defi}  For $a\in C\ell_{1,2}$, if there exists  an element  $b\in C\ell_{1,2}$ such that $ab=ba=1$ then  $b$ is called the inverse of $a$ and denoted by $b=a^{-1}$.
\end{defi}

Proposition \ref{uinverse} implies that the inverse of $a$ is unique. By Proposition \ref{properties}, we have

\begin{pro}\label{inverse}
 $a\in C\ell_{1,2}$ is invertible if and only if $P(a)\neq 0$ and in this case
$$a^{-1}= \frac{(N(a)-2T(a)\be_7)}{P(a)}\bar{a}.$$
\end{pro}

\begin{exam}
	Let $a=1+\be_2+\be_4$. Then $P(a)=9$ and
		$$a^{-1}=\frac{1-\be_2-\be_4}{3}.$$
\end{exam}

\section{The Moore-Penrose inverse of elements in $C\ell_{1,2}$}\label{mpinvsec}

\qquad We recall that the Moore-Penrose inverse of a real matrix $A$ is the unique  real matrix $X$ satisfying the following equations:
$$AXA=A,\,XAX=X,\,(AX)^T=AX,\,(XA)^T=XA,$$
 We denote the Moore-Penrose inverse of $A$ by $A^+$.

Note that $L(a')=L(a)^T$.  By  Theorem \ref{ringiso} and the concept of  Moore-Penrose inverse of real matrices,  we have the following lemma.

\begin{lem}\label{lemueq}
Let $a\in C\ell_{1,2}$. Then there is a unique $x\in C\ell_{1,2}$ satisfying the following equations:
\begin{equation}\label{mpeqs} axa=a,\,xax=x, \, (ax)'=ax, (xa)'=xa. \end{equation}
\end{lem}

\begin{defi}
	Let $a\in C\ell_{1,2}$. The unique solution $x\in C\ell_{1,2}$ of (\ref{mpeqs}) is called the Moore-Penrose inverse of $a$.
\end{defi}

Obviously, $0^+=0$ and   $a^+=a^{-1}$ if $P(a)\neq 0$.

Let
\begin{equation}\label{zcl12}Z(C\ell_{1,2})=\{a\in C\ell_{1,2}: P(a)=0 \}.\end{equation}

\vspace{1mm}

In order to find the  Moore-Penrose inverse of $a\in Z(C\ell_{1,2})-\{0\}$, we need the following lemma.

\begin{lem} \label{lem3.2}
Let $a\in C\ell_{1,2}-\{0\}$. Then

 \begin{itemize}
	\item[(1)] The following two equations \begin{equation} \label{ew1} ax=0,\,\,
		\, a'ax=0\end{equation} have the same solutions;
	\item[(2)] The following two equations \begin{equation} \label{ew2} xa'=0,\,\,
		\, xa'a=0\end{equation} have the same solutions.
\end{itemize}
	\end{lem}

\begin{proof}
	By Theorem \ref{ringiso}, if $ax=0$ then 	$\overrightarrow{ax}=L(a)\overrightarrow{x}=0$.  If $a'ax=0$ then 	$\overrightarrow{a'ax}=L(a)^TL(a)\overrightarrow{x}=0$. It is obvious that $L(a)\overrightarrow{x}=0$  and $L(a)^TL(a)\overrightarrow{x}=0$ have the same solutions. This proves (1). Similarly, we can prove (2).
\end{proof}
For $a\in Z(C\ell_{1,2})-\{0\}$,  we define    $$T_1=a_0a_1-a_2a_3-a_4a_5+a_6a_7;$$ $$T_3=a_0a_3+a_1a_2+a_4a_7+a_5a_6;$$
$$T_5=a_0a_5+a_1a_4- a_2a_7-a_3a_6.$$
By direct computation, we have the following Proposition.
\begin{pro}\label{pro3.1ad} For $a\in Z(C\ell_{1,2})-\{0\}$, we  have
 \begin{equation}\label{aap}a'a=\sum_{t=0}^{7}a_t^2+2T_1\be_1+2T_3\be_3+2T_5\be_5.\end{equation}
Let $v=v_0+v_1\be_1+v_3\be_3+v_5\be_5$. Then $v'=v$ and  \begin{equation}\label{vp} v'v=v_0^2+v_1^2+v_3^2+v_5^2+2v_0(v_1\be_1+v_3\be_3+v_5\be_5).\end{equation}
\end{pro}

\begin{lem} \label{lemT3} For $a\in Z(C\ell_{1,2})-\{0\}$ with $P(a)=0$, let \begin{equation}K=a_0^2+a_2^2+a_4^2+a_6^2.\end{equation} Then we have
\begin{equation} T_1^2+T_3^2+T_5^2=K^2. \end{equation}
\end{lem}
\begin{proof}
		If $P(a)=0$ then $N(a)=0$ and $T(a)=0$. Thus $$K=a_0^2+a_2^2+a_4^2+a_6^2=a_1^2+a_3^2+a_5^2+a_7^2=\frac{1}{2}\sum_{t=0}^{7}a_t^2$$ and $$a_0a_7+a_2a_5=a_1a_6+a_3a_4.$$
	So we have $$(a_0a_7+a_2a_5)^2+(a_1a_6+a_3a_4)^2=2(a_0a_7+a_2a_5)(a_1a_6+a_3a_4)$$and  $$K^2=(a_0^2+a_2^2+a_4^2+a_6^2)(a_1^2+a_3^2+a_5^2+a_7^2).$$
	Let $Q=K^2-(T_1^2+T_3^2+T_5^2)$. By direct computation, we have 	
$$Q=(a_0a_7+a_2a_5)^2+(a_1a_6+a_3a_4)^2-2(a_0a_7+a_2a_5)(a_1a_6+a_3a_4)=0.$$
	\end{proof}

\begin{lem} \label{lem3.3} For $a\in Z(C\ell_{1,2})-\{0\}$ with $P(a)=0$ and $K=a_0^2+a_2^2+a_4^2+a_6^2$, let  $$x=\frac{a'}{4K}.$$
Then \begin{equation} axa=a,\,xax=x, \, (ax)'=ax, (xa)'=xa. \end{equation}
\end{lem}

\begin{proof}Since $ax=\frac{aa'}{4K}$ and $xa=\frac{a'a}{4K}$, by Proposition \ref{properties}, we have  \begin{equation}(ax)'=ax,
	\,\,\, (xa)'=xa.\end{equation} 	
If $P(a)=0$ then $N(a)=0$ and $T(a)=0$. Thus $K=\frac{1}{2}\sum_{t=0}^{7}a_t^2$  and $$a_0a_7+a_2a_5=a_1a_6+a_3a_4.$$
Let $v=a'a$. By Proposition \ref{pro3.1ad} and Lemma  \ref{lem3.3} we have $v=v'=2K+2T_1\be_1+2T_3\be_3+2T_5\be_5$  and $$v'v=4K^2+4T_1^2+4T_3^2+4T_5^2+4K(2T_1\be_1+2T_3\be_3+2T_5\be_5)=8K^2+4K(a'a-2K).$$
 It follows from Lemma \ref{lemT3} that
 \begin{equation}\label{eqimpt}v'v=(a'a)(a'a)=4Ka'a.\end{equation}
By Lemma \ref{lem3.2}, $a(xa-1)=0$ is equivalent to \begin{equation}\label{qqeq1}a'a(\frac{a'a}{4K}-1)=0
\end{equation} and  $(xa-1)x=0$ is equivalent to \begin{equation}\label{qqeq2}(\frac{a'a}{4K}-1)a'a=0.
\end{equation} Equations (\ref{qqeq1}) and  (\ref{qqeq2}) are just (\ref{eqimpt}). Thus we have $axa=a,\,xax=x$.
\end{proof}

By Lemmas \ref{lemueq} and \ref{lem3.3}, we have the following theorem

\begin{thm} \label{thm} Let $a\in C\ell_{1,2}$. Then $$a^+=\left\{
	\begin{array}{ll}
		0, & \hbox{ if } a=0; \\
		\frac{(N(a)-2T(a)\be_7)}{P(a)}\bar{a}, & \hbox{ if } P(a)\neq 0; \\
		\frac{a'}{4(a_0^2+a_2^2+a_4^2+a_6^2)}, & \hbox{ if } a\neq 0 \hbox{ and } P(a)=0. \\
	\end{array}%
	\right.$$
	\end{thm}

\begin{exam}
	Let $a=\be_1+\be_2$. Then $P(a)=0$ and
	$$a^{+}=\frac{\be_1-\be_2}{4}.$$
\end{exam}

	By Theorem \ref{ringiso} and Lemma \ref{lemueq}, we have
\begin{pro}
	\begin{equation}\label{laplus}L(a^+)=L(a)^+\end{equation}
	and
	\begin{equation}\label{raplus}R(a^+)=R(a)^+.\end{equation}
\end{pro}

\section{Linear equation $axb=d$}\label{eqsec}

\qquad The following lemma is well known in matrix theory.
\begin{lem}\label{general} Let $A\in \br^{m\times n},  b\in
	\br^{m}$. Then the linear equation $Ax=b$ has a solution if and only if $AA^{+}b=b$, furthermore, the general solution is $$x=A^{+}b+(E_n-A^{+}A)y, \forall y\in \br^{n}.$$
\end{lem}

We are ready to consider the linear equation $axb=d$ in Clifford algebra $C\ell_{1,2}$. It is obvious that we have the following proposition.

\begin{pro}
 If $a,b\in  C\ell_{1,2}$ and both of them are invertible then $$x=a^{-1}db^{-1}$$ is the unique solution of $axb=d$.
 \end{pro}
   So we assume that $a$ and $b$ are not invertible in what follows.

In order to solve some linear equations in Clifford algebra $C\ell_{1,2}$, we  need the following Proposition.

\begin{pro}\label{mpprop}  Let $a,b\in C\ell_{1,2}$.  Then
 \begin{itemize}
	\item[(1)]	$L(a)L(a^+)L(a)=L(a), L(a^+)L(a)L(a^+)=L(a^+),L(a)L(a^+)=\big(L(a)L(a^+)\big)^T,L(a^+)L(a)=\big(L(a^+)L(a)\big)^T;$
	\item[(2)] $R(a)R(a^+)R(a)=R(a), R(a^+)R(a)R(a^+)=R(a^+),R(a)R(a^+)=\big(R(a)R(a^+)\big)^T,R(a^+)R(a)=\big(R(a^+)R(a)\big)^T;$
	\item[(3)] $\big(L(a)R(b)\big)^+=L(a^+)R(b^+).$
\end{itemize}
\end{pro}

\begin{proof}
	By  Proposition \ref{pro2.2} and Lemma \ref{lemueq}, we have (1) and (2).
	
    Let $A=L(a)R(b)$ and $X=L(a^+)R(b^+)$. Note that $L(a)R(b)=R(b)L(a),\forall a,b\in C\ell_{1,2}.$ By Proposition \ref{pro2.2}, we have
\begin{align*}
    AXA&=(L(a)R(b))(L(a^+)R(b^+))(L(a)R(b))\\
       &=L(a)L(a^+)L(a)R(b)R(b^+)R(b)\\
        &=L(a)R(b)=A.
\end{align*}
    Similarly, we have $XAX=X,(AX)^T=AX,(XA)^T=XA$. According to the definition of the Moore-Penrose, we have
   $(L(a)R(b))^+=L(a^+)R(b^+)$. This proves (3).
\end{proof}

\begin{thm}
Let $a,b\in Z(C\ell_{1,2})-\{0\}$ and $d\in C\ell_{1,2}$. Then the equation $axb=d$ is solvable if and only if
\begin{equation}\label{cond4.2}
\frac{aa'db'b}{16(a_0^2+a_2^2+a_4^2+a_6^2)(b_0^2+b_2^2+b_4^2+b_6^2)}=d,
\end{equation} in which case all the solutions are given by
\begin{equation}\label{sol4.2}x=\frac{a'db'}{16(a_0^2+a_2^2+a_4^2+a_6^2)(b_0^2+b_2^2+b_4^2+b_6^2)}+y-\frac{a'aybb'}{16(a_0^2+a_2^2+a_4^2+a_6^2)(b_0^2+b_2^2+b_4^2+b_6^2)},
\forall \  y\in C\ell_{1,2}.\end{equation}
	\end{thm}
\begin{proof}
	It is obvious that $axb=d$ is equivalent to $L(a)R(b)\overrightarrow{x}=\overrightarrow{d}$.  By Lemma \ref{general} $axb=d$ is  solvable
	if and only if  $$L(a)R(b)\big(L(a)R(b)\big)^+\overrightarrow{d}=\overrightarrow{d}.$$  Returning to Clifford algebra form  by Proposition \ref{mpprop}, we have $aa^+db^+b=d$. That is (\ref{cond4.2}).   By Lemma \ref{general},  the general solution is $$\overrightarrow{x}=\big(L(a)R(b)\big)^+\overrightarrow{d}+\Big(E_8-\big(L(a)R(b)\big)^+L(a)R(b)\Big)\overrightarrow{y},\forall y\in C\ell_{1,2}.$$  Hence the general solution can be expressed as$$x=a^+db^++y-a^+aybb^+,\,\forall y\in C\ell_{1,2}.$$
	That is  (\ref{sol4.2}).
	This concludes the proof.
\end{proof}

Similarly, we have the following corollaries.

\begin{cor}Let
	$a\in Z(C\ell_{1,2})-\{0\}$. Then the equation $ax=d$ is solvable if and only if $$\frac{aa'}{4(a_0^2+a_2^2+a_4^2+a_6^2)}d=d,$$ in which case all the solutions are given by $$x=\frac{a'd}{4(a_0^2+a_2^2+a_4^2+a_6^2)}+y-\frac{a'a}{4(a_0^2+a_2^2+a_4^2+a_6^2)}y ,\,\forall \  y\in
	C\ell_{1,2}.$$
\end{cor}

\begin{cor}Let
	$b\in Z(C\ell_{1,2})-\{0\}$. Then the equation $xb=d$ is solvable if and only if $$\frac{db'b}{4(b_0^2+b_2^2+b_4^2+b_6^2)}=d,$$ in which case all the solutions are given by $$x=\frac{db'}{4(b_0^2+b_2^2+b_4^2+b_6^2)}+y-\frac{ybb'}{4(b_0^2+b_2^2+b_4^2+b_6^2)},\,\forall \  y\in
	C\ell_{1,2}.$$
\end{cor}

We provide some examples as follows.
\begin{exam}
	Let $a=1+\be_1,\, b=\be_6+\be_7,\, d=1+\be_1+\be_6+\be_7$. Then
	$$a^+=\frac{1+\be_1}{4},\,\,b^+=\frac{-\be_6-\be_7}{4}, \,\, aa^+=a^+a=bb^+=b^+b=\frac{1+\be_1}{2},\,\, aa^+db^+b=d.$$
	This case belongs to Theorem 4.1.
	The  solutions of $axb=d$ are given by
	$$x=\frac{1+\be_1-\be_6-\be_7}{4}+y-\frac{(1+\be_1)y(1+\be_1)}{4},\,\forall y\in C\ell_{1,2}.$$
\end{exam} 	
	\begin{exam}
		Let $a=\be_1+\be_2,\,d=\be_1+\be_2+\be_5+\be_6$. Then
	$$a^{+}=\frac{\be_1-\be_2}{4},\,\,a^+a=\frac{1+\be_3}{2}, \,\,aa^+=\frac{1-\be_3}{2}, \, aa^+d=d .$$
			 This case belongs to Corollary 4.1.  The  solutions of $ax=d$ are given by
	$$x=\frac{1+\be_3+\be_4+\be_7}{2}+y-\frac{(1+\be_3)y}{2},\,\forall y\in C\ell_{1,2}.$$
\end{exam} 	
	\begin{exam}
	 	Let $b=\be_6+\be_7,\,d=\be_2-\be_3+\be_4-\be_5$. Then
	 $$b^{+}=\frac{-\be_6-\be_7}{4},\,\,b^+b=bb^+=\frac{1+\be_1}{2}, \,\, db^+b=d .$$  This case belongs to Corollary 4.2.  The  solutions of $xb=d$ are given by
	$$x=\frac{-\be_2+\be_3+\be_4-\be_5}{2}+y-\frac{y(1+\be_1)}{2},\,\forall y\in C\ell_{1,2}.$$
\end{exam}

\section{Similarity}\label{simsec}

\qquad It is well known that two quaternions are similar if and only if they have the same norm and real part. Such relationships were extended to other algebra systems.  For example, Yildiz and Kosal etc. have studied comsimilarity    and  semisimilarity  of split quaternions in \cite{kosal,onder}. In this section, we will introduce the concept of similarity of two elements in $C\ell_{1,2}$ and obtain the necessary and sufficient conditions for them to be similar.

\begin{defi}
	We say that $a,b\in C\ell_{1,2}$  are similar
	if and only if there exists an element $q\in C\ell_{1,2}-Z(C\ell_{1,2})$ such that $qa=bq$.
\end{defi}
Let $$P_1=\{a\in C\ell_{1,2}:P(a)=1\}.$$
  For  $a\in C\ell_{1,2}-Z(C\ell_{1,2})$, we define  $\phi(a)$ to be a map acting on  $C\ell_{1,2}$  as follows: \begin{equation}\phi(a)(x):=axa^{-1},x\in C\ell_{1,2}.\end{equation}
Note that $$\phi(a)(x)=\phi(ta)(x),\,\, \forall  a\in C\ell_{1,2}-Z(C\ell_{1,2}), t \in \br-\{0\}.$$
Since  $P(ta)=t^4P(a)$, we may assume  that $a\in P_1$ when we define $\phi(a)$.

\begin{pro}\label{prosim}
 $Cre(C\ell_{1,2})$ and $Cim(C\ell_{1,2})$ are two invariant subspaces of $\phi(a)$.
\end{pro}
\begin{proof}
It is easy to verify the following properties: \begin{equation}(s+te_7)Cim(C\ell_{1,2})\subset Cim(C\ell_{1,2}),\forall s,t\in \br.\end{equation}
  \begin{equation} a\be_i\bar{a}\subset Cim(C\ell_{1,2}),\forall \be_i,i=1,\cdots, 6.\end{equation}
If $a\in P_1$ then  $$a^{-1}= (N(a)-2T(a)\be_7)\bar{a}.$$
Let $A=L(a)R((N(a)-2T(a)\be_7)\bar{a})$.  Then we have  \begin{equation}\label{phidets} \overrightarrow{\phi(a)(x)}=L(a)R(a^{-1})\overrightarrow{x}=L(a)R((N(a)-2T(a)\be_7)\bar{a})\overrightarrow{x}=A\overrightarrow{x}.
\end{equation}
Noting that $$a(s+te_7)a^{-1}=(s+te_7),\forall s,t\in \br,$$ we have
$$A=\left(\begin{array}{ccc}
	1 &0&0\\
	0 & S(a)&0\\
	0&0&1
\end{array}\right),$$
 where $S(a)\in GL(6,\br)$  determined by $a$. In fact $$\left(\begin{array}{c}
 	0\\
 	S(a)\\
 	0
 \end{array}\right)=(\overrightarrow{a\be_1a^{-1}},\cdots,\overrightarrow{a\be_6a^{-1}}).$$
This implies that $Cre(C\ell_{1,2})$ and $Cim(C\ell_{1,2})$ are two invariant subspaces of $\phi(a)$.
\end{proof}

Obviously, we have the following proposition.
\begin{pro}\label{lemsim}
	Let $a\in Cent(C\ell_{1,2})$. Then $a,b \in C\ell_{1,2}$ are similar if and only if $a=b.$
\end{pro}

\begin{lem}\label{lem51} If    $Cim(a)$ is invertible and  $$N(Cim(a))=N(Cim(b)),\,\,\,T(Cim(a))=T(Cim(b))$$ then at least
	one of the four elements
	$Cim(a)e_t+e_tCim(b),\, t=0,1,2,3$ is invertible.
\end{lem}

\begin{proof} Suppose that all the four elements
$w_t=Cim(a)e_t+e_tCim(b),\, t=0,1,2,3$ are not invertible. Let
$Cim(a)=a_1e_1+a_2e_2+a_3e_3+a_4e_4+a_5e_5+a_6e_6$ and
$Cim(b)=b_1e_1+b_2e_2+b_3e_3+b_4e_4+b_5e_5+b_6e_6$. Then $$P(w_t)=0,\,\,t=0,1,2,3.$$
By $P(w_0)=0$  we have
\begin{equation}\label{e40}(a_1 + b_1)(a_6 + b_6) + (a_3 + b_3)(a_4 + b_4)=(a_2 + b_2)(a_5 + b_5),\end{equation}
\begin{equation}\label{e41}(a_1 + b_1)^2 +(a_3 + b_3)^2+(a_5 + b_5)^2=(a_2 + b_2)^2+(a_4 + b_4)^2+ (a_6 + b_6)^2.\end{equation}
By $P(w_1)=0$  we have
\begin{equation}\label{e42}(a_1 + b_1)(a_6 + b_6)+(a_3 - b_3)(a_4 - b_4)=(a_2-b_2)(a_5-b_5),\end{equation}
\begin{equation}\label{e43}(a_1 + b_1)^2+(a_3 - b_3)^2+(a_5 - b_5)^2=(a_2 - b_2)^2 +(a_4 - b_4)^2+(a_6 + b_6)^2.\end{equation}
By $P(w_2)=0$  we have
\begin{equation}\label{e44}(a_1 - b_1)(a_6 - b_6) + (a_3 - b_3)(a_4 - b_4)=(a_2 + b_2)(a_5 + b_5),\end{equation}
\begin{equation}\label{e45} (a_1 - b_1)^2+(a_3 - b_3)^2+(a_5 + b_5)^2=(a_2 + b_2)^2+ (a_4 - b_4)^2+(a_6 - b_6)^2.\end{equation}
By $P(w_3)=0$  we have
\begin{equation}\label{e46}(a_1 - b_1)(a_6 - b_6)+(a_3 + b_3)(a_4 + b_4)=(a_2 - b_2)(a_5 - b_5)\end{equation}
\begin{equation}\label{e47}(a_1 - b_1)^2+(a_3 + b_3)^2+(a_5 - b_5)^2 =(a_2 - b_2)^2+(a_4 + b_4)^2 +(a_6 - b_6)^2.\end{equation}
By equations(\ref{e41}),(\ref{e43}),(\ref{e45}) and (\ref{e47}), we can deduce that \begin{equation}\label{eqlem1}a_1b_1=a_6b_6,\,a_2b_2=a_5b_5,\,a_3b_3=a_4b_4.\end{equation}
Thus (\ref{e41}) becomes
$$a_1^2+a_3^2+a_5^2+b_1^2+b_3^2+b_5^2=a_2^2+a_4^2+a_6^2+b_2^2+b_4^2+b_6^2.$$
That is $N(Cim(a))=-N(Cim(b))$.
By $N(Cim(a))=N(Cim(b))$, we have that $$N(Cim(a))=N(Cim(b))=0.$$
By equations(\ref{e40}),(\ref{e42}),(\ref{e44}) and (\ref{e46}), we can deduce that
\begin{equation}\label{ablemeq2}a_1b_6+b_1a_6=0,a_2b_5+b_2a_5=0,a_3b_4+b_3a_4=0.\end{equation}
Thus (\ref{e40}) becomes
$$a_1a_6+b_1b_6+a_3a_4+b_3b_4=a_2a_5+b_2b_5.$$
That is  $-T(Cim(a))=T(Cim(b))$.
By $T(Cim(a))=T(Cim(b))$, we have that $$T(Cim(a))=T(Cim(b))=0.$$
Hence we have $$P(Cim(a))=P(Cim(b))=0,$$
 which contradicts the assumption that $Cim(a)$ is invertible.
 \end{proof}

 If   $Cim(a)=Cim(b)$ then $$Cim(a)x=xCim(b),\,\,\forall x\in \br-\{0\}.$$

\begin{lem}\label{lem52} If   $Cim(a)\neq Cim(b)$ and $Cim(a)\neq 0,\,\,Cim(b)\neq 0$ and $P(Cim(a))=P(Cim(b))=0$ then at least
	one of the four elements
	$Cim(a)e_t+e_tCim(b),\, t=0,1,2,3$ is invertible.
\end{lem}

\begin{proof} Suppose that all the four elements
	$w_t=Cim(a)e_t+e_tCim(b),\, t=0,1,2,3$ are not invertible. Let
	$Cim(a)=a_1e_1+a_2e_2+a_3e_3+a_4e_4+a_5e_5+a_6e_6$ and
	$Cim(b)=b_1e_1+b_2e_2+b_3e_3+b_4e_4+b_5e_5+b_6e_6$.
	Then as in Lemma \ref{lem51}, we have
\begin{equation}\label{1eqlem1}a_1b_1=a_6b_6,\,a_2b_2=a_5b_5,\,a_3b_3=a_4b_4.\end{equation}
and
\begin{equation}\label{1ablemeq2}a_1b_6+b_1a_6=0,\,a_2b_5+b_2a_5=0,\,a_3b_4+b_3a_4=0.\end{equation}
By $a_1b_1=a_6b_6$ and $a_1b_6+b_1a_6=0$, we have
$$a_1(b_1^2+b_6^2)=0,\,a_6(b_1^2+b_6^2)=0$$ and
$$b_1(a_1^2+a_6^2)=0,\,b_6(a_1^2+a_6^2)=0.$$
Thus we have the following inference:
\begin{equation}\label{ed16}a_1^2+a_6^2\neq 0\Rightarrow b_1=b_6=0;\quad b_1^2+b_6^2\neq 0\Rightarrow a_1=a_6=0.
\end{equation}
Similarly we have
\begin{equation}\label{ed25}a_2^2+a_5^2\neq 0\Rightarrow b_2=b_5=0;\quad b_2^2+b_5^2\neq 0\Rightarrow a_2=a_5=0;
\end{equation}
 \begin{equation}\label{ed34}a_3^2+a_4^2\neq 0\Rightarrow b_3=b_4=0;\quad b_3^2+b_4^2\neq 0\Rightarrow a_3=a_4=0.
 \end{equation}
 Since $Cim(a)\neq 0$, without  loss of generality, we suppose that $a_1\neq 0$. Then we have $b_1=b_6=0$.  Note that
 $N(Cim(b))=0,\,\,T(Cim(b))=0$. Hence we have
 $$b_3^2+b_5^2=b_2^2+b_4^2,\,\, b_2b_5=b_3b_4.$$
 From the above equations, we have $$(b_5^2-b_4^2)(b_3^2+b_5^2)=0,\,\,(b_5^2-b_4^2)(b_2^2+b_4^2)=0.$$
   Since $Cim(b)\neq 0$, we may assume that $b_2\neq 0$.  It follows from (\ref{ed25}) that $a_2=a_5=0$ and  $$b_5^2=b_4^2,\,\, b_3^2=b_2^2.$$
  So we deduce that $b_3\neq 0$, which implies that $a_3=a_4=0$ by (\ref{ed34}).
   Note that
 $N(Cim(a))=a_2^2+a_4^2+a_6^2-a_1^2-a_3^2-a_5^2=0,\,\,T(Cim(a))=a_2a_5-a_1a_6-a_3a_4=0$. That is  $$a_1^2=a_6^2, \,\,a_1a_6=0.$$
 The above two equations is a contradiction, which concludes the proof.
\end{proof}

\begin{thm}\label{thmsim}
	Two Clifford numbers $a,b \in C\ell_{1,2}-Cent(C\ell_{1,2})$ are similar if and only if
	$$Cre (a)=Cre (b), \,N(a)=N(b) \, \mbox{and}\,\,\,T(a)=T(b).$$
\end{thm}


\begin{proof}
	If  $a,b \in C\ell_{1,2}$ are similar then there exists an element $q\in C\ell_{1,2}-Z(C\ell_{1,2})$ such that $qa=bq$. Since $b=\phi(q)(a)$, by Proposition \ref{prosim}, we have $Cre (a)=Cre (b)$.
	Also we have $$q\bar{q}a\bar{a}=qa\overline{qa}=bq\overline{bq}=q\bar{q}b\bar{b}.$$
	Thus $$q\bar{q}(a\bar{a}-b\bar{b})=0.$$ Since $P(q\bar{q})=P(q)^2\neq 0$, $q\bar{q}$ is invertible. Hence we have $a\bar{a}=b\bar{b}.$  Note that  $$a\bar{a}=N(a)+2T(a)e_7,\,\,b\bar{b}=N(b)+2T(b)e_7.$$  This implies that  $N(a)=N(b)$ and $T(a)=T(b)$. This proves the necessity.
	
	To prove the sufficiency, we need find an element $x\in C\ell_{1,2}-Z(C\ell_{1,2})$ such that $$(Cre(a)+Cim(a))x=ax=xb=x(Cre(b)+Cim(b))$$ under the conditions $Cre (a)=Cre (b), \, N(a)=N(b)$ and $T(a)=T(b).$
	This is equivalent to  \begin{equation}\label{simeq}Cim(a)x=xCim(b), \mbox{ for some } x\in C\ell_{1,2}-Z(C\ell_{1,2}).\end{equation}
	Noting that $$Cim(a)Cim(a)=Cre(a)^2-N(a)-2T(a)e_7,$$
	we have $$Cim(a)Cim(a)=Cim(b)Cim(b)\in Cent(C\ell_{1,2}).$$
	Hence $$x=Cim(a)p+pCim(b),\,\,\forall p\in C\ell_{1,2}$$
	are solutions of equation (\ref{simeq}).
	Especially, $$x_i=Cim(a)\be_i+\be_iCim(b),\,\,i=0,\cdots,4$$
are solutions of equation (\ref{simeq}).
Applying Lemmas \ref{lem51} and \ref{lem52} concludes the proof.

\end{proof}

{\bf Acknowledgements.}\quad This work is supported by Natural Science Foundation of China (11871379), Innovation Project of Department of Education of Guangdong Province (2018KTSCX231) and Key project of  Natural Science Foundation  of Guangdong Province Universities (2019KZDXM025).

\noindent Wensheng Cao,  Ronglan Zheng, Huihui Cao\\
School of Mathematics and Computational Science,\\
Wuyi University\\
Jiangmen, Guangdong,  P.R. China\\
e-mail: {\tt wenscao@aliyun.com}

\end{document}